\documentclass{amsart}
\usepackage[english]{babel}
\usepackage{amsmath,amsthm,amssymb, amsfonts}

\usepackage{xfrac}
\usepackage{hyperref}
\usepackage{tikz}

\newtheorem{theorem}{Theorem}
\newtheorem{lemma}{Lemma}
\newtheorem{assumption}{Assumption}

\theoremstyle{definition}
\newtheorem{definition}{Definition}

\theoremstyle{remark}
\newtheorem{remark}{Remark}

\newcommand{\mathd}{\mathrm{d}}
\newcommand{\R}{\mathbb{R}}

\begin{document}
	
	\title[White noise and supercritical bifurcation]{The impact of white noise on a supercritical bifurcation in the Swift-Hohenberg equation}
	
	\author{Luigi Amedeo Bianchi}
	\address{Luigi Amedeo Bianchi \\
		Università di Trento\\
		Dipartimento di Matematica\\
		Via Sommarive, 14\\
		38123 Povo (Trento)\\
		Italy}
	
	\email{luigiamedeo.bianchi@unitn.it}
	
	\author{Dirk Bl\"omker}
	\address{Dirk Bl\"omker\\ Universit\"at Augsburg\\
		Universit\"atsstra\ss{}e 14\\
		86159 Augsburg\\
		Germany}
	\email{dirk.bloemker@math.uni-augsburg.de}

\begin{abstract}
 We consider the impact of additive Gaussian white noise 
 on a supercritical pitchfork bifurcation 
 in an unbounded domain.
 As an example we focus on the stochastic
 Swift-Hohenberg equation with polynomial nonlinearity.
 Here we identify the order where small noise first impacts the bifurcation.
 Using an approximation via modulation equations, we provide a tool to analyse how 
 the noise influences the dynamics close to a change of stability.
\end{abstract}



\maketitle

\section{Introduction}

In this paper we intend to identify the main impact of an additive Gaussian white  
noise on the dynamics close to or at a change of stability described 
by a stochastic partial differential equation with polynomial nonlinearity.
For this we will study the reduction of the essential dynamics close to the bifurcation 
via amplitude or modulation equations.
Surprisingly, and in contrast to the strong nonlinear interaction 
of finitely many Fourier modes,
in all our results the additive noise 
does not add any additional terms to the modulation equation,
its nonlinear interaction always disappears via averaging effects
and it just shows up as an additive forcing in the amplitude equation.

In order to keep the paper short and to focus on the main results,
we do not aim to prove all error estimates in full technical details, 
but we always state how they can be proven.

As a first problem we consider 
the following stochastic Swift-Hohenberg equation on $\R^{+}\times \R$
\begin{equation}
\label{SH}
\partial_t u = - (1 + \Delta)^2 u + \nu u^2 - u^3 
+ \varepsilon^{3/2} \partial_t \widetilde{W,}
\end{equation}
where $\widetilde{W}$ is a standard cylindrical Wiener process, 
i.e.~$\partial_t \widetilde{W}$ models space-time white noise. This equation was introduced first in the seminal paper by Swift and Hohenberg,~\cite{SwiHoh1977PRA}, where they already discussed the importance of random fluctuations in the context of Rayleigh B\'enard convection.

The operator $ - (1 + \Delta)^2$ is a non-positive self-adjoint operator 
with spectrum $(-\infty,0]$. 
As we do not have an additional linear term in the equation \eqref{SH}, 
we are exactly at criticality, where the spectrum of the linear operator is 
non-positive, but it contains $0$, which in our case formally corresponds 
to the complex eigenfunction $e^{ix}$.
The parameter $\nu$ in front of the quadratic term in the equation 
does not change the linearised operator. It will only determine 
the shape of the bifurcation.

In~\cite{BiaBloSch2019CMP} we discussed the equation with $\nu=0$ 
and an additional linear term
in the weakly nonlinear regime close to bifurcation,
and we comment on that in more detail below.  

In the deterministic case the dynamics of \eqref{SH} 
and its importance in pattern formation 
was studied in numerous publications. 
See for example 
\cite{BurDaw2012SJADS,BurKno2006PRE,CroHoh1993RMP,HilMetBorDew1995PRE,KirOM2014JoMP},
where also many examples of a formal derivation of amplitude equations are found.

Rescaling the equation, we will see in our main result that solutions 
are given by a slow modulation
of the dominating solution (or pattern) $e^{ix}$, that is
\[
u(t,x)=\varepsilon A(\varepsilon^2t,\varepsilon x) e^{ix} + \text{c.c.}
\]
where c.c.\ denotes the complex conjugate. We denote by $T=\varepsilon^2 t$ the slow time 
and by $X=\varepsilon x$ the rescaled 'slow' space variable.
In the case of \eqref{SH} we will see that the complex-valued amplitude $A$ function solves 
\[
\partial_T A = 4\partial_X^2 A - \Big(3-\frac{38}{9}\nu^2\Big)|A|^2A + \eta
\]
where $\eta$ is a complex-valued space-time white noise.
So the 
presence of the quadratic term in Swift-Hohenberg 
can change the strength of the cubic in the amplitude equation.
It formally can even make the sign of the cubic positive, 
for $\nu >\sqrt{\sfrac{27}{38}}$.
Although our analysis carries through even in this case, 
this leads to an unstable cubic in the amplitude equation 
and would allow for a blow up of solutions in finite time.
Our analysis in that case only holds up to times where the solution of the 
amplitude equation is still of order~$1$. 
Similar results on a bounded domain, 
where the amplitude equation is just a SDE, were derived in~\cite{KleBloMoh2014ZAMP}.
 
In \cite{BiaBloSch2019CMP} we studied the classical Swift-Hohenberg equation without a 
quadratic nonlinearity (i.e. with $\nu=0$) but with an additional linear term
\[
\partial_t u = - (1 + \Delta)^2 u + \mu\varepsilon^2 u - u^3 
+ \varepsilon^{3/2} \partial_t \widetilde{W}.
\]
Here the spectrum of the linear operator is $(-\infty,\mu\varepsilon^2]$ 
and thus changes stability at $\mu=0$, 
which means we have a bifurcation here. 
Further analysis in the deterministic case would reveal, that it is a classical supercritical 
(i.e., forward) pitchfork bifurcation,
where non-trivial stationary states are present only for $\mu>0$. 

Moreover, in \cite{BiaBloSch2019CMP} we showed that the amplitude
in this case solves 
\[
\partial_T A = 4\partial_X^2 A + \nu A- 3|A|^2A + \eta.
\]
For the effect of a simple scalar valued forcing, 
which is constant in space, see~\cite{MohBloKle2013SJMA}.

In this paper in Section~\ref{sec:quintic},
we  also briefly consider the Swift-Hohenberg equation with 
a quintic nonlinearity 
\begin{equation}
\label{SH5}
\partial_t u = 
- (1 + \Delta)^2 u 
+ \nu_2 \varepsilon^{1/2} u^2 
+ \nu_3 \varepsilon u^3 
- u^5 
+ \varepsilon \partial_t \widetilde{W.}
\end{equation}
As the analysis is quite similar to the cubic case,
we will keep the presentation very short here, 
and only focus on the main differences.

The advantage of adding the quintic is the following. 
In the  
setting of \eqref{SH} without the stable cubic 
in the case of a subcritical bifurcation, 
we would have a positive coefficient in front of the highest cubic nonlinear term
in the amplitude equation, which
thus leads to an equation that might blow up in finite time.
In contrast to that the additional quintic leads 
to a stable quintic in the amplitude equation, which prevents blow up.

Note that due to the quintic nonlinearity, 
we have a different scaling of the parameters
and the quadratic and cubic nonlinearities 
have to be small in order to not dominate the quintic close to bifurcation.
In the scaling 
\[
u(t,x)=\varepsilon^{1/2} A(\varepsilon^2 t,\varepsilon x) e^{ix} + \text{c.c.}
\]
we obtain the following equation for the complex amplitude 
\[
\partial_T A = 4\partial_X^2 A  + \Big(\frac{38}{9}\nu_2^2+3\nu_3\Big)|A|^2A -10|A|^4A   
+ \eta \;.
\]

If $\nu_2$ is sufficiently large when compared to $\nu_3$
or $\nu_3$ being positive, then the cubic is an unstable 
subcritical nonlinearity.
This means that, if we were to add a linear term $\nu_1\varepsilon^2u$ to \eqref{SH5} 
we would obtain also an additional $\nu_1 A$ in the amplitude equation. 
This equation has for $\nu_1=0$ a subcritical 
backward pitchfork bifurcation if $\nu_2$ 
is sufficiently large
so that the constant in front of the cubic is positive.

Let us also comment that we could also add a quartic nonlinearity 
$\nu_4 \varepsilon^{-1/2}u^4$, to~\eqref{SH5} 
which now leads to an additional quintic nonlinearity 
with positive coefficients in the amplitude equation.
On the expense of overwhelming technical difficulties 
one could now go to even higher order nonlinearities.

Surprisingly, in all our results the additive noise 
does not introduce any additional terms to the modulation equation,
it just appears as an additive forcing in the amplitude equation.
This is in contrast to the strong nonlinear interaction of Fourier modes
that, for example, leads to the appearance of cubic terms in the amplitude equation 
arising from a quadratic nonlinearity in  \eqref{SH5}.
We will however see that in this setting all the nonlinear interactions of noise 
terms actually vanish due to averaging effects.

The outline of the paper is as follows.
In the next Section \ref{sec:sol}, we briefly discuss 
the problem of existence and uniqueness of solutions
and mainly give references to methods that allow to prove this.
In Section \ref{sec:apb} we rely on estimates to identify the dominant 
Fourier modes, which are the ones around the wavenumbers $k\in\{0,\pm1,\pm2\}$ 
in Fourier space
and derive reduced equations for these modes by cutting out all small terms.
Using explicit averaging results based on It\^o formula in Section \ref{sec:aver},
we reduce the whole dynamics  to the wavenumbers close to $k=\pm1$ in Fourier space
and state in Section \ref{sec:final} the final result.
Assuming additional regularity of the dominant Fourier modes, 
we simplify the limiting equation in Section \ref{sec:limiting}.
In the final Section \ref{sec:quintic} we briefly comment 
on the changes necessary for the result in the quintic case.

\section{Solutions} 
\label{sec:sol}

Due to a lack of regularity of solutions due to the noise, 
we consider solutions to our SPDEs in the mild sense.
The mild formulation of~\eqref{SH} is given by
\begin{align*}
u(t)=e^{t L} u(0) 
&+ \int_0^t e^{(t-s)L} [\nu u^2 - u^3](s) ds 
\\& + \varepsilon^{3/2} \int_0^t e^{(t-s)L} d\widetilde{W}(s),
\end{align*}
where $e^{tL}$ is the semigroup generated by the operator $L=-(1+\Delta)^2$.
On the unbounded domain we can simply rely on the fact that the linear operator  
is diagonal in Fourier space
and define the  semigroup using the standard Fourier transform $\mathcal{F}f=\widehat{f}$.
For example,   
\[\widehat{Lf}(k) =-(1-|k|^2)^2\widehat{f}(k)\]
and for the semigroup 
\[\mathcal{F}[e^{t L}f](k)= \exp\{-(1-|k|^2)^2t\}\widehat{f}(k).
\]
We will now first rescale the equation and then comment on the existence of solutions 
for the rescaled equation further below.

\subsection*{Rescaling:}
Close to bifurcation we consider small solutions 
and follow the usual deterministic approach of modulation equations. 
We rescale small solutions to slow spatial and temporal scales via
\[
u(t,x)=\varepsilon v(\varepsilon^2 t, \varepsilon x)
\]
to obtain
\begin{equation}
\label{SH:resc}
\partial_T v = L_{\varepsilon} v + \varepsilon^{- 1} \nu v^2 - v^3 +
   \partial_T W, 
  \end{equation} 
with the rescaled operator 
$L_{\varepsilon} = - \varepsilon^{- 2}  (1 + \varepsilon^2 \Delta)^2$.

The noise strength is derived  
using the scaling property of the white noise or, equivalently, the scaling property of 
the Wiener process $\widetilde{W}$. 
Here $\partial_T W$ is again space-time white noise and $W$ a standard 
cylindrical Wiener process. Due to the rescaling $W$ and thus $\partial_T W$
depend path-wise on $\varepsilon$, but as they have the same law as 
$\widetilde{W}$ and $\partial_t\widetilde{W}$,
and we consider error estimates only in law, 
we ignore this dependence in the following.  

The mild formulation of \eqref{SH:resc} is given by
\begin{align}
\label{SH:rescmild} 
v(T)=e^{T L_{\varepsilon}} v(0) 
& + \int_0^T e^{(T-S)L_{\varepsilon}} [\varepsilon^{- 1}  \nu v^2 - v^3](S) dS 
\nonumber\\ & +   \int_0^T e^{(T-S)L_{\varepsilon}} d{W}(S).
\end{align}

We consider solutions in spaces $C^{0,\alpha}_\kappa$, the 
spaces of $\alpha$-H\"older continuous functions with slow polynomial growth at infinity:
 \begin{equation*}
C^{0,\alpha}_\kappa 
= \{u:\mathbb{R}\to\mathbb{R}\ :\ 
\sup\{L^{-\kappa}\|u\|_{C^{0,\alpha}([-L,L])} \ ; \ L>1 \}<\infty \}.
 \end{equation*}
 A more detailed discussion regarding these spaces can be found in~\cite{BiaBlo2016SPA}.
 
If we consider the stochastic convolution
 \begin{equation*}
 {W}_{L_{\varepsilon}}(T)=\int_0^T e^{(T-S)L_{\varepsilon}} d{W}(S)
 \end{equation*}
 we have the following uniform bound in the spaces $C^{0,\alpha}_\kappa$.
\begin{lemma}
\label{lem:SC}
 For all $\alpha \in(0,\frac12)$, $\kappa>0$, 
 the stochastic process is ${W}_{L_{\varepsilon}}$ has continuous paths in $C^{0,\alpha}_\kappa$ 
 and for all  $T>0$ and $p>1$, we have a constant such that for all $\varepsilon\in(0,1)$
\[  \mathbb{E} \sup_{[0,T]}\|{W}_{L_{\varepsilon}} \|^p_{C^{0,\alpha}_\kappa} \leq C
\]
\end{lemma}
The proof for this Lemma is quite long but at the same time fairly standard. 
It can be proven using exactly 
the same arguments as in the proof of Lemma 3 in~\cite{BiaBloSch2019CMP}.
There one considers first bounded spatial domains of length $2L$, 
and then carefully keeps track of the dependence of various constants on $L$.

For other type of maximal regularity results for the stochastic convolution, for instance in $L^p$ spaces, see
\cite{BrzPes2000SPPGNI,DaLun1998AANLCSFMNRLMA,vanVerWei2012AP}.

\begin{remark}
 Let us remark that ${W}_{L_{\varepsilon}}$ is actually more 
regular than stated in Lemma~\ref{lem:SC}.
It is H\"older-continuous with exponent $\alpha$ almost $1$. 
This is due to strong regularization of  the fourth 
order operator in the equation. But in the limit $\varepsilon\to0$ 
(see~\cite{BiaBlo2016SPA}) we lose this property 
and thus a uniform bound in $\varepsilon$ 
can only be established for H\"older exponents $\alpha<1/2$.
\end{remark}

In the rest of the paper, we always suppose that we have sufficiently 
smooth solutions such that the following statement holds.

\begin{assumption}
	\label{ass:existence}
	The rescaled equation \eqref{SH:resc} has a unique mild solution $u$, 
	which is a stochastic process with continuous paths 
	in $C^{0,\alpha}_\kappa$ for every $\kappa>0$ and $\alpha \in(0,\frac12)$.
\end{assumption}

\begin{remark}
Before moving on, let us remark that for fixed $\kappa$ and $\alpha$ the standard fixed point argument
for the existence and uniqueness of local mild solutions does not work,
as the nonlinearity is unbounded in the weight 
and the semigroup only improves regularity in
terms of the H\"older exponent.
\end{remark}

We state Assumption~\ref{ass:existence} as such, and not as a theorem, because its proof would be a paper of its own, so within this paper we just take existence for granted. At the same time, we are quite confident that such result holds: there are some fairly standard approaches we could follow to prove it.
Nevertheless this is quite a lot of work, as most results first establish the existence 
and uniqueness in a weaker topology and then lengthy regularity results are needed.

Among the first results on SPDEs on the whole real line in spatially 
weighted spaces are the ones of Peszat et. al. \cite{BrzPes1999SM,PesZab1997SPA} 
using mainly exponential weights but also stating results for polynomial weights.

The complex-valued stochastic Ginzburg-Landau Equation 
in a weighted $L^2$-space was studied in detail by Bl\"omker and Han \cite{BloHan2010SD},
but not with regularity in H\"older spaces, which was done in \cite{BiaBloSch2019CMP},
where also Swift-Hohenberg with $\nu=0$ was discussed.

For recent results on space-time-white noise in weighted Besov spaces
see for example R\"ockner, Zhu, and Zhu
\cite{RocZhuZhu2017JFA} or Mourrat and Weber\cite{MouWeb2017AP}.

Let us also mention a recent paper by Moinat and Weber \cite{MoiWeb2020CPAM}
that obtains for the dynamic $\Phi^4_3$ model local regularization on bounded subdomains 
in case of weaker bounds on the whole domain. Although the model they treat is real-valued 
the results should hold for the very similar complex Ginzburg Landau model. 
Moreover, this method should also apply to Swift-Hohenberg.

\section{A-priori bound} 
\label{sec:apb}

In this section, we show that Fourier modes around $\pm1$ 
(or $\pm1/\varepsilon$ for the rescaled equation) 
dominate the behaviour.

One can easily argue that the mild solution with
initial condition $v(0)$ of order~$1$ stays 
of order~$1$ at least for some time.
However, this is not obvious.
If we look at the rescaled equation \ref{SH:resc}, 
then  the quadratic term has an additional $1/\varepsilon$ in front,
and due to the rescaled solution 
and the semigroup being only of order~$1$, 
we do not immediately get a bound up to times of order~$1$, 
as a direct estimate yields only that the quadratic term is bounded by something of order~$1/\varepsilon$. 
In contrast to that, the cubic term does not cause any difficulties,
as everything is order one there. This is also reflected by the fact that we consider small solutions to the original equation~\eqref{SH}, 
so the cubic term should be always smaller than the quadratic one.

With this simple reasoning we can only hope to reach times of order~$\varepsilon$,
so we need a better estimate.

In order to restrict to regions around $k\approx\pm1/\varepsilon$ 
in Fourier space,
we consider smooth projectors $P_1$ for a given smooth Fourier kernel 
$q:\mathbb{R}\to[0,1]$ such that $q=1$ on the set of $k$ such 
that  $|k\pm1/\varepsilon|< \delta/\varepsilon$ 
and $q=0$ on $|k\pm1/\varepsilon|> \delta/\varepsilon+1$, for some $\delta<\delta_0\leq1/2$. 
Hence,
\[
P_1 f (x) = \frac{1}{2\pi} \int_{\mathbb{R}} \int_{\mathbb{R}} 
q(k)e^{ik(x-z)}dk f(z) dz\;. 
\]

Before we move on, 
let us spend a couple of words on some notation used in the rest of this paper. 
First, we give the following definition:
\begin{definition} \label{defO}
 We say that an $\varepsilon$-dependent event $E_\varepsilon$ has \emph{probability almost
		1} or \emph{high probability} if for every $p \geq 1$ there exists a positive constant $C_p$ such that 
	$P(E_\varepsilon)\geq 1 - C_p \varepsilon^p$.
\end{definition}

Secondly, let us discuss briefly our use of the $\mathcal{O}$ 
notation. In the following, we use it in two ways, in both cases considering $\varepsilon\to 0$ or just sufficiently small $\varepsilon$.
On one hand it means that the term is bounded 
for all $\varepsilon \in (0,\varepsilon_0)$ for some $\varepsilon_0>0$
up to an $\varepsilon$-independent
multiplicative constant (as we will see for the semigroups in the next paragraph).

On the other hand, for stochastic processes $w$ 
(e.g.~our solutions) we write 
$w=\mathcal{O}(\varepsilon^{\gamma})$ 
if for all $c>0$,
$\kappa>0$, and $\alpha \in (0,1/2)$ 
there is a constant $ C_{\alpha,\kappa,c}$ such that with
probability almost $1$ (see Definition~\ref{defO} above) we have
\begin{equation}
 \label{e:defO}
\sup_{T\in[0,T_0]}\|w(T)\|_{C^{0,\alpha}_{\kappa}} \leq  C_{\alpha,\kappa,c} \varepsilon^{\gamma-c}\;,
\end{equation}
again for all $\varepsilon \in (0,\varepsilon_0)$.

Note that the $c>0$ allows for small logarithmic corrections 
to the error bound, which frequently pop up when bounding stochastic convolutions pathwise.

\subsection*{First estimate:}
Note that in Fourier space around $k\approx\pm1/\varepsilon$ by looking at the eigenvalues 
we have  $L_{\varepsilon} \leq0$ 
and $L_{\varepsilon} \approx 0$,
but for $|k\pm1/\varepsilon| > \delta/\varepsilon$ 
we have   $L_{\varepsilon} \leq - C\varepsilon^{-2}$.
This also carries over to the semigroups,
so if $P_1$ projects to the $\delta/\varepsilon$-neighbourhoods
around $k=\pm1/\varepsilon$ in Fourier space, then
we have
\begin{equation}
\label{e:SGP1}
 P_1 e^{TL_\varepsilon} = \mathcal{O}(1) 
\quad\text{and}\quad 
(I-P_1) e^{TL_\varepsilon} = \mathcal{O}(e^{ - cT/\varepsilon^{2}}).
\end{equation}
This result is straightforward to verify, 
as the operators are all diagonal in Fourier space.

Using the mild formulation, we now aim to show that, for $v_1 := P_1v$, 
\[
v = v_1 + \mathcal{O} (\varepsilon)
\]

Recall that by Lemma~\ref{lem:SC} we have  $W_{L_\varepsilon}=\mathcal{O}(1)$,
but we can improve it with the following:

\begin{lemma}\label{lem:other_projections} 
For the two projections $P_1$ and $I-P_1$ of the stochastic convolution 
$W_{L_\varepsilon}$, we have
\[P_1 W_{L_\varepsilon}=\mathcal{O}(1) 
\quad\text{and}\quad
(I-P_1) W_{L_\varepsilon}=\mathcal{O}(\varepsilon).\]
\end{lemma}
\begin{proof}[Idea of Proof]
In order to prove this Lemma, one can follow the same ideas as 
in Lemma \ref{lem:SC}. 
The key point is that due to \eqref{e:SGP1} the integrand in one case is still order~1, 
while it is small in the other.
\end{proof}

Assume that $v(0)=\mathcal{O}(1)$. Then up to times where $v=\mathcal{O}(1)$
we directly obtain from \eqref{SH:rescmild} 
\[
v(T)=  \mathcal{O}(1) 
+ \int_0^T  \mathcal{O}(\varepsilon^{- 1}) dS 
\]
which is not sufficient for times $T$ of order~$1$.
We need to split $v$ in order to obtain a better estimate.
First  using the bounds on the semigroup from \eqref{e:SGP1} we can show 
that 
\[ (I-P_1) e^{TL_\varepsilon} v(0) = \mathcal{O}(e^{ - cT/\varepsilon^{2}}).
\]
For the other terms in the mild formulation, we use a similar estimate, together with the 
results for the 
stochastic convolution from the previous Lemma~\ref{lem:other_projections} 
in order to obtain that
\begin{multline*}
(I-P_1) v(T)= 
\mathcal{O}(e^{ - cT/\varepsilon^{2}}) + \mathcal{O}( \varepsilon )\\
+ \int_0^T e^{ - c(T-S)/\varepsilon^{2}}) \mathcal{O}(\varepsilon^{- 1})  dS .
\end{multline*}
Thus up to times where $v=\mathcal{O}(1)$
we have
\[  (I-P_1) v(T) = (I-P_1) e^{TL_\varepsilon} v(0) + \mathcal{O}(\varepsilon).
\]
After a short logarithmic time $t_\varepsilon>0$,
we have  
\[(I-P_1)v(t_\varepsilon\varepsilon^2)=\mathcal{O}(\varepsilon).
\] 
Moreover, if we assume that 
$P_1v(0)=\mathcal{O}(1)$ and $(I-P_1)v(0)=\mathcal{O}(\varepsilon)$, then 
$(I-P_1) v=\mathcal{O}(\varepsilon)$ as long as $v=\mathcal{O}(1)$.

Let us now turn to a bound on $v_1=P_1v$.
Here we rely crucially on the fact that 
$P_1 (P_1v)^2 =0$, if $\delta$ is small, 
so that 
\[
P_1 (v_1+\mathcal{O}(\varepsilon))^2 = \mathcal{O}(\varepsilon).
\]

If we now assume that $v = v_1 + \mathcal{O}  (\varepsilon)$, then
the quadratic term in the nonlinearity is always $\mathcal{O}(1)$
and
we obtain from \eqref{SH:rescmild} that $v_1= \mathcal{O}(1)$ up to some times of order~1.
To be more precise, the dominant estimate is to the type
\begin{multline*}
 \|v_1(T)\| \leq C \|v_1(0)\|+  C \int_0^T (\|v_1(S)\|+\|v_1(S)\|^3) dS \\
 + \text{ error terms}
\end{multline*}
Thus we find a time of order one, such that $v_1$ remains of order one
if $v_1(0)$ is of order one.

We have thus sketched the proof of the following theorem:
\begin{theorem}[Attractivity]\label{thm:attractivity}
Consider a solution $v$ of \eqref{SH:resc}
with initial conditions of order $\mathcal{O}(1)$,
then for a suitable logarithmic time $t_\varepsilon$ the solution is bounded by
\begin{equation}\label{eq:boundv1}
v_1(\varepsilon^2 t_\varepsilon):=  P_1v(\varepsilon^2 t_\varepsilon) = \mathcal{O} (1)\ \text{and}\  (1-P_1)v(\varepsilon^2 t_\varepsilon) = \mathcal{O}  (\varepsilon).
\end{equation}
Additionally,
if we have this bound in (\ref{eq:boundv1}) for the 
initial conditions
(i.e.~$v_1(0)= \mathcal{O}(1)$ and $(1-P_1)v(0) = \mathcal{O}  (\varepsilon)$),
then up to some constant time $T_0>0$
\begin{equation}\label{eq:boundIC}
v_1= \mathcal{O} (1)\quad\text{and}\quad (1-P_1)v = \mathcal{O}  (\varepsilon).
\end{equation}
\end{theorem}

In particular, the Fourier modes around $\pm1/\varepsilon$ dominate the behaviour 
close to the bifurcation.

\begin{remark}
Let us remark that with the estimates for the mild solution we 
cannot rely on any stability of the cubic. 
From the final result we will see later that $T_0$ might be small 
if the cubic in the amplitude equation has a positive sign in front of the nonlinearity:  
in this case the cubic is actually unstable and allows for blow up in finite time (but of order one).
On the other hand, if the sign is negative one can show global bounds and thus $T_0$ can be arbitrary. 
\end{remark}

In the following results, we will assume that
a short period of time already has passed, so that the bound~\eqref{eq:boundv1} is already effective, 
and we can start a solution in the setting of the 
second statement of Theorem \ref{thm:attractivity}.
To me more precise, we will assume in the following estimates 
that we have a solution $v$ such that~\eqref{eq:boundIC} 
holds for some $T_0>0$.

\begin{remark}
At the moment each Fourier mode in $v_1$ 
can have the same order of magnitude, 
but we can even show that they are given by a modulated wave
\[
v_1(T,X)= A(T,X) e^{iX/\epsilon}+\text{c.c.}
\]
for an amplitude $A$ having a little bit of regularity.
In that case the Fourier transform of $A$ decays
for wave-number $|k|\to \infty$, and thus 
the Fourier modes of $v_1$ are 
slightly more 
concentrated in Fourier space around the Fourier modes $\pm1/\epsilon$.
See Figure~\ref{fig:fig1} for a sketch. We will come back to this point 
in section \ref{sec:limiting}, when we discuss the final approximating 
equation and identify the terms in it.
\end{remark}

\begin{figure}
\begin{center}
\begin{tikzpicture}[scale=1,>=latex]
%Achse
 \draw[black,arrows=->,line width=1.3pt] (0,-0.5) to (0,1);
 \draw[black,arrows=->,line width=1.3pt] (-5,0) to (5,0);
         \draw [black] (-2.1,-0.25) node {\bf $-1/\varepsilon$};
         \draw [black] (2,-0.25) node {\bf $1/\varepsilon$};
        %Dichte
        \draw [red,line width=1.3pt] plot [smooth] coordinates {(-4,0) (-3,0.05) (-2.3, 0.1) (-2.1,0.8) (-2,1) (-1.9,0.8) (-1.7,0.1) (-1,0.05)(0,0.01)(1,0.05) (1.7,0.1) (1.9,0.8)(2,1)(2.1,0.8) (2.3,0.1) (3,0.05) (4,0)};
       \draw[black,line width=1.3pt] (2,-0.1) to (2,0.1);
       \draw[black,line width=1.3pt] (-2,-0.1) to (-2,0.1);
       \draw [red] (-3,0.5) node {\bf $|\mathcal{F} u|$};
 \draw[black,arrows=<->,line width=1pt] (1.7,-0.5) to (2.3,-0.5); 
  \draw[black,arrows=<->,line width=1pt] (-1.7,-0.5) to (-2.3,-0.5); 
  \draw [black] (2.7,-0.5) node {\bf $\pm\delta/\varepsilon$};
   \draw [black] (-1.3,-0.5) node {\bf $\pm\delta/\varepsilon$};
    \end{tikzpicture}
\end{center}
\caption{Fourier transform of $u(x)=A(x)e^{ix/\varepsilon}+c.c.$ for a not too rough amplitude $A$.}
\label{fig:fig1}
\end{figure}
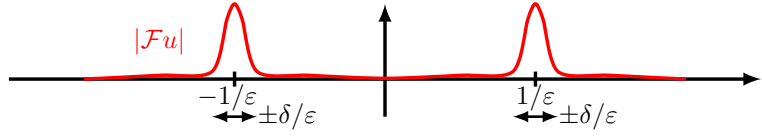

\subsection*{Higher order ansatz:}
In order to identify the higher order terms of order $\mathcal{O}(\varepsilon)$, 
we further split $v$ as follows: 
\begin{equation}
\label{e:ansatz}
 v = v_1 + \varepsilon v_0 + \varepsilon v_2 + \varepsilon R, 
\end{equation}
with $v_1=P_1v$, as before, concentrated in Fourier space on modes 
$k$ such that 
$\left| k \pm \frac{1}{\varepsilon}\right| < \frac{\delta}{\varepsilon}$.
For the two new terms we also use smooth Fourier projections $P_0$ and $P_2$
such that 
$v_0=\varepsilon^{-1}P_0 v$ is concentrated in Fourier space 
on modes with $| k | < \frac{2\delta}{\varepsilon}$, 
and $v_2=\varepsilon^{-1}P_2v$ is concentrated on $\left| k \pm \frac{2}{\varepsilon}
\right| < \frac{2 \delta}{\varepsilon}$.
Note that is contrast to $v_1$ we also rescale $v_0$ and $v_2$ by a factor 
$\varepsilon^{-1}$, so that they are of order $1$.
Finally, $R$ just collects all the remaining terms.

\begin{remark}
	It might seem strange at the first glance that we choose 
	different radii for the regions in Fourier space around $\pm1$ 
	and for the ones around $0$ and $\pm2$, 
	but the reason we are considering the projections $P_0$ and $P_2$ 
	is to take care of the second order (i.e., the quadratic) terms. 
	But, when we square a term, in Fourier space we also double the size of its support, 
	hence, double the radius.
	
	In other terms, we want $(P_2+P_0)v_1^2=v_1^2$, or equivalently 
	$(I-P_2-P_0)v^2_1 =0$, so we do not want to cut away some parts of $v_1^2$, 
	which would happen with smaller balls in Fourier space around $0$ and $2$.
\end{remark}

\begin{remark}
Note that the $R$ in the ansatz~\eqref{e:ansatz} is simply
$R = \varepsilon^{-1}(I-P_1-P_2-P_0)v$. 
Here we cannot show that this term is  smaller than $\mathcal{O} (\varepsilon)$, 
as it contains 
the term $(I-P_1-P_2-P_0)W_{L_\varepsilon}$, which is $\mathcal{O} (\varepsilon )$, 
from the stochastic convolution
and we cannot show that it is smaller.  
\end{remark}

We will now use also $W_k = P_kW$, for $k=0,1,2$, to shorten the notation a bit.

Let us first check the equation for $v_1$. 
Simply projecting \eqref{SH:rescmild} with $P_1$
we see that $v_1$ is the mild solution of 
\begin{equation} \label{e:v1equation1}
\partial_T v_1 
= L_{\varepsilon} v_1 + \nu\varepsilon^{- 1} P_1  v^2 - P_1
   v^3 + \partial_T W_1\;,
\end{equation}
which we would also obtain by projection \eqref{SH:resc} directly.
Note that we have a bounded linear operator 
$L_{\varepsilon} P_1=  P_1L_{\varepsilon}=\mathcal{O}(1)$.

Now, by the ansatz \eqref{e:ansatz} we obtain for the cubic
\[
 P_1  v^3 = P_1 (v_1)^3 + \mathcal{O}(\varepsilon)\;,
\]
and for the quadratic term
\[
P_1  v^2 = P_1 (v_1)^2 + 2 \varepsilon P_1 (v_1 (v_0 + v_2 + R))
+ \mathcal{O}(\varepsilon^2)\;.
\]
Using the properties of the projectors in Fourier space, we have $P_1 (v_1)^2 =0 $ and $P_1 (v_1 R)=0$  
so that 
\[ \varepsilon^{-1}P_1  v^2 =  2  P_1 (v_1 v_0)  + 2  P_1 (v_1 v_2)
+ \mathcal{O}(\varepsilon)\;.
\]

We can plug this into~\eqref{e:v1equation1} to finally derive
\begin{multline}
\label{e:SHreduced}
\partial_T v_1 
= L_{\varepsilon} v_1 + 
2 \nu P_1 (v_1 v_0)  + 2 \nu P_1 (v_1 v_2)\\
- P_1(v_1)^3 + \mathcal{O}(\varepsilon) + \partial_T W_1\;.
\end{multline}

We would like, however, to have an equation in $v_1$ only, so we need to understand the behaviour of the two mixed products $v_1v_0$ and $v_1v_2$. This is the topic for the next section.

\section{Averaging}
\label{sec:aver}

Let's go on with the  terms $v_0$ and $v_2$ 
appearing in the ansatz \eqref{e:ansatz} above. 
The aim of this section is to show that when we consider the two products $v_1v_k$ for $k=0,2$ in~\eqref{e:SHreduced}, their leading order terms are in $v_1$ only.
From the rescaled Swift-Hohenberg equation in \eqref{SH:resc} or \eqref{SH:rescmild}
we have by projection with $P_0$
\[ 
\partial_T v_0 
= L_{\varepsilon} v_0 
+ \varepsilon^{- 2}\nu P_0  v^2 - \varepsilon^{- 1} P_0 v^3 
+ \varepsilon^{- 1} \partial_T W_0 
\]
with a bounded linear operator 
$P_0L_{\varepsilon}=L_{\varepsilon}P_0 \approx O (\varepsilon^{- 2})$.
Recall also that $v_0=\varepsilon^{-1}P_0v$, which makes the 
coefficients different from the equation for $v_1$.

As before we expand the nonlinear terms using \eqref{e:ansatz}
together with the properties of Fourier projections 
to obtain 
\[ \partial_T v_0 
= L_{\varepsilon} v_0 
+ \varepsilon^{- 2} P_0 \nu v_1^2 +\mathcal{O}(\varepsilon^{-1})
+ \varepsilon^{- 1} \partial_T W_0\;
\]
Moreover,
\[ 
\partial_T v_2 = L_{\varepsilon} v_2 + \varepsilon^{- 2} P_2 \nu v_1^2 +
  \mathcal{O}(\varepsilon^{-1}) + \varepsilon^{- 1} \partial_T W_2,
   \]
analogous to the previous one for $v_0$. 

Note again that in the two equations above 
for $v_k$, $k\in\{0,2\}$, the linear 
operators are bounded, but also large, as
$P_kL_\varepsilon=L_\varepsilon P_k=\mathcal{O}(\epsilon^{-2})$.
Nevertheless, for fixed $\varepsilon>0$ we can consider strong solutions of these equations 
in order to apply It\^o formula.

\begin{remark}
Note that in the mild formulation of the two equations above
for both $v_0$ and $v_2$, we have for the stochastic convolution 
\[\varepsilon^{-1}\int_0^T e^{(T-S)P_kL_\varepsilon}dW_k(S)=\mathcal{O}(1),
\] 
so one could conjecture that the noise has an $\mathcal{O}(1)$ 
contribution to $v_0$ and $v_2$.
But it is an Ornstein-Uhlenbeck process on the fast-time scale, 
so we will see below that its contribution in lowest order 
is actually negligible due to averaging.     
\end{remark}

We proceed by an explicit averaging result via It\^o formula.
The two operators $P_kL_{\varepsilon}$, $k = 0, 2$, are bounded and invertible.  
Furthermore, we can use It\^o formula and note that we get no correction terms in it, 
since the noise terms are independent. 
We thus obtain 
\begin{align*}
  \mathd [v_1 L_{\varepsilon}^{- 1} v_k]  = & L_{\varepsilon}^{- 1} v_k \mathd
  v_1 + v_1 L^{- 1}_{\varepsilon} \mathd v_k\\
   = & L_{\varepsilon}^{- 1} v_k  (L_{\varepsilon} v_1 +  \mathcal{O}(1) )\mathd t
 + L_{\varepsilon}^{- 1} v_k \mathd W_1\\
    & + v_1 L_{\varepsilon}^{- 1} [L_{\varepsilon} v_k + \varepsilon^{- 2}
  P_k \nu v_1^2 + \mathcal{O}(\varepsilon^{- 1})] \mathd t
 \\ &+ \varepsilon^{-1} v_1 L_{\varepsilon}^{- 1} \mathd W_k .
\end{align*}
Since the operator 
$L_{\varepsilon}^{-1}P_k=\mathcal{O}(\varepsilon^2)$, we can
identify the leading order terms. 
Only $v_1 v_k \mathd t$  and
$\nu \varepsilon^{- 2} L_{\varepsilon}^{- 1} P_k v_1^2 \mathd t$ 
 are of order 1. All other terms are small in $\varepsilon$.

So we can rewrite the previous equations to obtain  
\begin{equation}\label{e:products_identities} 
\int_0^T v_1 v_k \mathd t 
+ \nu  \int_0^T v_1 \varepsilon^{- 2} L_{\varepsilon}^{- 1} P_k v_1^2 \mathd t  
= \mathcal{O}(\varepsilon).
\end{equation}
We have thus identified for both cases $k=0$ and $k=2$ 
the leading order terms in~\eqref{e:SHreduced}.

Let us briefly remark here that in Section~\ref{sec:limiting}, when we identify explicitly the terms 
in the limiting equation, we will see that $\varepsilon^{- 2} L_{\varepsilon}^{- 1} P_k$ 
 can be replaced by suitable  constants.

We now look at equation \eqref{e:SHreduced} for $v_1$: 
\begin{multline*} 
\partial_T v_1 
= L_{\varepsilon} v_1 + 2\nu (P_1 (v_1 v_0) + P_1 (v_1v_2)) \\
- P_1 (v_1)^3 + \mathcal{O}(\varepsilon)  + \partial_T W_1\;,
\end{multline*}
in integral from 
in  order to plug in the averaging results from~\eqref{e:products_identities} 
to replace the terms including $v_0$ and $v_2$.
We obtain 
\begin{multline}
\label{e:finint}
 v_1(T) = v_1(0)+ \mathcal{O}(\varepsilon)  
+  W_1(T)
  \\ + \int_0^T \Big[L_{\varepsilon} v_1 - 2\nu^2 P_1 v_1 \varepsilon^{- 2} L_{\varepsilon}^{- 1} (P_0+P_2) v_1^2 - P_1 (v_1)^3 \Big]dS
\;.
\end{multline}
Neglecting the error term gives the final result
\begin{multline}
\label{e:final}
\partial_T v_1 
= L_{\varepsilon} v_1 
- 2\nu^2 P_1 v_1 \varepsilon^{- 2} L_{\varepsilon}^{- 1} (P_0+P_2) v_1^2 \\
- P_1 (v_1)^3  
+ \partial_T W_1\;.
\end{multline}
Let us remark that this approximation still depends on $\varepsilon$, but we will see later in Section 
\ref{sec:limiting} that in the setting of modulation equations 
we can further approximate it by an  $\varepsilon$-independent Ginzburg-Landau equation.
But for our purpose this approximation is sufficient, as it shows that the noise only appears as an additive forcing 
in the equation for the dominating modes. We will summarise our results in the next section and draw some conclusions.

\section{Final Result} 
\label{sec:final}

As we have now the limiting equation~\eqref{e:final} for $v_1$,
we can prove the following theorem:

\begin{theorem}
\label{thm:mainapprox}
Consider a solution $v$ of the rescaled Swift-Hohenberg 
equation~\eqref{SH:resc} such that the bound~\eqref{eq:boundIC} of Theorem~\ref{thm:attractivity} holds up to some $T_0>0$,  
that is, $v_1=\mathcal{O}(1)$  and $(I-P_1)v=\mathcal{O}(\varepsilon)$.

If $w$ is a solution of~\eqref{e:final}, then 
\[P_1v- w = \mathcal{O}(\varepsilon) \]
in the sense given in~\eqref{e:defO}.
\end{theorem}

\begin{proof}[Idea of proof]
In the previous section  we saw in estimate \eqref{e:finint} that $P_1v$ 
satisfies equation \eqref{e:final}  with an additional small residual.

To remove the residual from \eqref{e:final},
we rely on the continuous dependence of the solution on an additive forcing.
This is a fairly standard argument, but, once again, quite long and technical 
if all the details are provided. We do not give it here.

Let us only give the key steps in order to motivate the error bound.  
Consider \eqref{e:finint} for $v_1=P_1v$ and the time-integrated 
version of \eqref{e:final} for $w$. With this we build 
the following equation for
the 
difference 
\begin{multline*}
 e(T) := v_1(T)- w(T) =  \mathcal{O}(\varepsilon) +
 \int_0^T   L_{\varepsilon} e ds 
  \\ 
  - 2\nu^2  \int_0^T  \Big[  P_1 v_1 \varepsilon^{- 2} L_{\varepsilon}^{- 1} (P_0+P_2) v_1^2
 - w \varepsilon^{- 2} L_{\varepsilon}^{- 1} (P_0+P_2) w^2 \Big]ds
 \\ -   \int_0^T   \Big[ P_1 (v_1)^3  - P_1 w^3 
 \Big]dS\\
\;.
\end{multline*}
Now we have a deterministic equation, where all cubic terms 
can be estimated by a global Lipschitz property, which follows from the   
fact that $v_1=\mathcal{O}(1)$ by \eqref{eq:boundIC} and that we can similarly show $w=\mathcal{O}(1)$.
Thus a lengthy computation involving Gronwall's Lemma shows that  
$e(T)= \mathcal{O}(\varepsilon)$.
\end{proof}

\begin{remark}[Global estimates]
Let us remark, without proof, that when the nonlinearity in \eqref{e:final} is a  stable cubic 
then
we can check that the solution of \eqref{e:final} exists for all times $T_0>0$ and is order $\mathcal{O}(1)$.
The assumption of Theorem~\ref{thm:mainapprox} remains true for any $T_0>0$,
and we obtain that even for large times of order one
the Fourier modes around $k=\pm1$ dominate the solution of \eqref{SH:resc},
and their dynamics is given by~\eqref{e:final}.
\end{remark}

\begin{remark}[No additional impact of noise]
Our main result is now a negative one. 
We consider Swift-Hohenberg in a scaling where small additive noise has an effect
on the dynamics. If we take smaller noises, then
we would see no contribution at all in the limiting equation. 

But even in our scaling, although there is strong nonlinear interaction of Fourier modes,
the impact of the noise is actually quite limited, due to the effect of averaging. 
The noise only 
appears as an additive forcing on the dominant modes, 
which is exactly the noise put into the original equation.
There is no further effect.
\end{remark}

Let us stress that even additive noise still might have a significant impact on the dynamics.
See for example \cite{BloHaiPav2015} where a degenerate noise was able to change the stability of a trivial solution in a Swift-Hohenberg equation. Nevertheless, we do not expect this here with full additive space-time white noise.

\begin{remark}[Dominant Pattern]
Without analysing the dynamics of \eqref{e:final} in detail, we can already draw the implication that all the essential dynamics
of our Swift-Hohenberg equation is given by $v_1$, which is 
concentrated in Fourier space around $\pm1$, or $\pm1/\epsilon$ for the rescaled equation. Thus we expect have solutions given by modulated pattern of an underlying $2\pi$-periodic pattern. 
\end{remark}

\section{Identifying the limit}
\label{sec:limiting}

The main result, Theorem \ref{thm:mainapprox}, already shows that the noise  
in the abstract modulation equation \eqref{e:final} appears only as an additive forcing.
Here we want to present some results on how to identify the terms in the equation  
\eqref{e:final}  in the limit $\varepsilon\to0$.

We will use the ansatz, suggested by the modulation equation approach,
\[
v_1(T,X) = A(T,X)e^{iX/\varepsilon} + \text{c.c.} 
\]
with some smoothness of $A$. Let us remark, that a more detailed analysis 
as used in Theorem \ref{thm:attractivity}
for the attractivity result should justify that after some time this result is
typically true for bounded solutions of \eqref{SH:resc}. 

Note that the smoothness of $A$ is an assumption here.
In space we cannot assume more than weighted H\"older spaces 
with exponent strictly less than $1/2$. 
See for example \cite{BiaBloSch2019CMP} or one on the 
many other results on the (complex or real) Ginzburg-Landau 
(also called Allen-Cahn or $\Phi^4_3$-model) in 1D, some of which we have mentioned in Section~\ref{sec:sol}.

The crucial term that needs enough smoothness is the linear operator.  
If we have that $A\in C^4_\kappa$ is order one, then    
we can evaluate directly as done by Kirrmann, Mielke, and Schneider in~\cite{KirSchMie1992PRSESM} 
\[ 
L_{\varepsilon} v_1(T,X) = 4\partial_X^2 A(T,X)e^{iX/\varepsilon} + \text{c.c.} + \mathcal{O}(\varepsilon).
\]
In the theory of deterministic modulation equation there are numerous 
results, which need less regularity than \cite{KirSchMie1992PRSESM}.
See for example Part IV of \cite{SchUec2017} also for many other examples in this direction. 
But still they need derivatives and moreover $A$ to be uniformly bounded in space.

This is in the stochastic case, however, too much regularity to ask for, 
so we need to take a different approach.
In the setting of weighted H\"older-regularities, using the mild 
formulation of equation \eqref{e:final} we can replace 
the semigroups of the Swift-Hohenberg operator $L_\varepsilon$ acting on $v_1$ 
by the semigroup generated by $4\partial_X^2$ acting on $A$,
which is the mild version of the statement we are looking for.
This is rigorously proven in the exchange lemmas in \cite{BiaBloSch2019CMP}.

For the noise, we also have to treat the mild formulation 
of the modulation equation \eqref{e:final}.
In there we have the stochastic convolution 
\[
(W_1)_{L_\varepsilon}(T)
= P_1W_{L_\varepsilon}(T) = P_1\int_0^T e^{(T-S)L_\varepsilon}dW(S).
\]
It was proven in \cite{BiaBlo2016SPA} that we have 
\[P_1 W_{L_\varepsilon}(T,X) 
\approx \mathcal{W}_{4\partial_X^2}(T,X) e^{iX/\varepsilon}
\]
for a complex-valued standard cylindrical Wiener 
process $\mathcal{W}$ that consists of a rescaling 
of the Fourier modes of $W$ acting on the dominant modes around $k=1$,
or $k=1/\varepsilon$ in the rescaled version. 
Moreover, one can write $\mathcal{W}$ explicitly in terms of $W$.
Finally, $\eta=\partial_T\mathcal{W}$ is complex valued space-time white noise.

Let us now turn to the nonlinear terms.
For the simple cubic term we obtain, by expanding the cube,
\[
 - P_1 (v_1)^3 \approx  -3 A|A|^2  e^{iX/\varepsilon} + \text{c.c.} 
\]
The previous is actually not an identity, but only an approximation, as 
\[
(1-P_1)A|A|^2  e^{iX/\varepsilon} \neq 0.
\] 
This is, on the other hand, a contribution to the non-dominant modes, 
which are small by Theorem \ref{thm:mainapprox}.

For the other cubic terms, let us start by considering the one with the projection $P_0$.
In the following we are neglecting error terms given by 
contributions to the non-dominant Fourier modes.
For example  $(I-P_0) |A|^2$ is non-zero, 
but small nonetheless, due to the regularity of $A$. 
We obtain
\begin{align*}
\varepsilon^{- 2} L_{\varepsilon}^{- 1} P_0 v_1^2(T,X)
 &= 2\varepsilon^{- 2} L_{\varepsilon}^{- 1} P_0 |A|^2(T,X) \\
  &= -2 (1+(\varepsilon^2\partial_X^2))^{- 2} P_0 |A|^2(T,X)\\
 &= -2|A|^2(T,X).
\end{align*}
For the step where we replaced  $ L_{\varepsilon}^{- 1}$  
using the eigenvalues of the operator, we can easily see that 
\[(1+(\varepsilon^2\partial_X^2)^{- 1} P_0 = 1+\mathcal{O}(\delta).
\]
Recall that  $(1+(\varepsilon^2\partial_X^2))^{-2} 1 = 1$. 
But, using a little bit of regularity of $A$, we can improve this result to an error term that is small in $\varepsilon$. 
Thus finally,
\[
- \nu^2 P_1 v_1 \varepsilon^{- 2} L_{\varepsilon}^{- 1} P_0 v_1^2 =   2 \nu^2 A|A|^2(T,X) e^{iX/\varepsilon} +\text{c.c.} 
\]

Similarly, we have for the cubic term involving $P_2$,
\begin{equation*}
\begin{split}
\lefteqn{\varepsilon^{- 2} L_{\varepsilon}^{- 1} P_2 v_1^2(T,X)}\\ 
& =  - (1+(\varepsilon^2\partial_X^2))^{-2} P_2 A^2(T,X)e^{i2X/\varepsilon} +\text{c.c.} \\
& = - \frac19  A^2(T,X)e^{i2X/\varepsilon} +\text{c.c.} 
\end{split}
\end{equation*}
The main difference with respect to the previous term is due to the different constant. This can be seen by the fact that 
$(1+(\varepsilon^2\partial_X^2))^{-2}e^{i2X/\varepsilon} = \frac19$.
We finally obtain
\begin{multline*}
- 2\nu^2 P_1 v_1 \varepsilon^{- 2} L_{\varepsilon}^{- 1} (P_0+P_2) v_1^2 \\ 
 =2(2+1/9) \nu^2 A|A|^2(T,X) e^{iX/\varepsilon} +\text{c.c.} \;.
\end{multline*}

Collecting all cubic terms together with the result on the semigroups and the stochastic convolution,
we finally obtain the mild formulation of the  
Ginzburg-Landau equation
\[
\partial_T A = 4 \partial_X^2 A - \Big(3- \frac{38}{9} \nu^2\Big) A|A|^2 + \eta.
\]

So finally this Ginzburg-Landau equation gives the dynamics of the amplitude of the modulated pattern that dominates the behaviour of the Swift-Hohenberg equation.
Thus the properties of the change of dynamics when varying $\nu$ 
should reflect the dynamics of this Ginzburg-Landau equation.

If there is no quadratic nonlinearity in the Swift-Hohenberg 
equation (i.e. $\nu=0$), then we have a stable cubic 
$-3A|A|^2$ in Ginzburg-Landau, and due to the presence of noise $A$  should be centred around zero on the order of the noise strength.

Turning on the quadratic nonlinearity weakens this stability and lets $A$ be bigger, until a critical threshold, where 
the stability breaks down.

It would be  interesting to characterize this bifurcation 
in terms of random attractors of invariant measures. However there seems to be no invariant measure available yet
in the setting we are working in.

For random attractors the situation, even on bounded domains, is not that clear, and it might be that for additive noise the random attractor is anyway always a single stable stationary solution. 
See for example \cite{FlaGesSch2017PTRF} for a quite general result in the case of SDEs, and \cite{BiaBloYan2016N} for a highly degenerate noise 
in a Swift-Hohenberg equation on an unbounded domain.

In terms of invariant measures, we suppose that, 
similarly to \cite{BloHai2004CMP}, one should be able to extend 
our approximation result to invariant measures, 
provided the Ginzburg-Landau equation is ergodic 
in the setting we are interested in. 
Then the qualitative changes in the dynamics, as described heuristically above, could be seen in qualitative changes of the invariant measure. 

\section{Quintic case}
\label{sec:quintic}

Here we comment briefly on the modifications necessary 
in the quintic case, stated in~\eqref{SH5} and rewritten here for ease of reference:
\begin{equation*}
\partial_t u = 
- (1 + \Delta)^2 u 
+ \nu_2 \varepsilon^{1/2} u^2 
+ \nu_3 \varepsilon u^3 
- u^5 
+ \varepsilon \partial_t \widetilde{W.}
\end{equation*}

Let us begin by saying that we do not discuss the existence of solutions.
Similarly to the cubic case~\eqref{SH}, this can be done using standard methods,
and we assume here that an analogue to Assumption~\ref{ass:existence} holds also for \eqref{SH5}.
 
The scaling 
\[
u(t,x)=\varepsilon^{1/2} v(\varepsilon^2 t,\varepsilon x)
\]
in \eqref{SH5} yields 
\begin{equation*}
%\label{SH5:resc}
\partial_T v = L_\varepsilon v 
+ \frac{1}{\varepsilon}\nu_2 v^2 
+ \nu_3 v^3 
- v^5 
+ \partial_T W.
\end{equation*}

\subsection*{Attractivity:} The attractivity result is now very similar, 
as apart from the quintic, we have exactly the same terms in the equation.
We only have to note that
\[
(v_1+\mathcal{O}(\varepsilon))^5 = (v_1)^5+\mathcal{O}(\varepsilon)\;.
\]
Thus the quintic (as the cubic) does not change any of the estimates and we can assume 
that $v_1$ is also dominant. In other words, 
\[ 
v=v_1+\mathcal{O}(\varepsilon).
\]  

\subsection*{Equation for \texorpdfstring{$v_1$}{v1}:}
Similar to what we had in~\eqref{e:v1equation1} for the cubic, $v_1$ solves 
\[
\partial_T v_1 
= L_{\varepsilon} v_1 + \nu_2\varepsilon^{- 1} P_1  v^2 +\nu_3 P_1
   v^3 - P_1v^5 + \partial_T W_1\;.
\]
and thus expanding the powers and using as before that $P_1v_1^2=0$ and $P_1(I-P_2-P_0) =0$ yields
\begin{equation}
\label{e:v_1SH5}
 \partial_T v_1 
= L_{\varepsilon} v_1 + 2 \nu_2P_1 v_1(v_2+v_0)  +\nu_3 P_1v_1^3 - P_1v_1^5 + \mathcal{O}(\varepsilon) + \partial_T W_1\;.
\end{equation}

\subsection*{Averaging:}
In a similar way as the equation for $v_1$ we derive (using $\varepsilon v_k = P_k v$) 
from \eqref{SH5} that (with $k=0$ and $k=2$)
\[ 
\partial_T v_k 
= L_{\varepsilon} v_k 
+ \varepsilon^{- 2}\nu_2 P_k  v^2 + \varepsilon^{- 1}\nu_2 P_k v^3-  \varepsilon^{- 1}P_k v^5
+ \varepsilon^{- 1} \partial_T W_k\;.
\]
As we did earlier in the cubic case, we expand the nonlinear terms
to obtain 
\[ \partial_T v_k 
= L_{\varepsilon} v_k 
+ \varepsilon^{- 2} \nu_2P_k  v_1^2 +\mathcal{O}(\varepsilon^{-1})
+ \varepsilon^{- 1} \partial_T W_k\;.
\]
Now the averaging of the quadratic terms in the quintic case \eqref{e:v_1SH5}
 is exactly the same as for the cubic case \eqref{SH} and we obtain
 \begin{multline}
\label{e:final5}
\partial_T v_1 
= L_{\varepsilon} v_1 
- 2\nu_2^2 P_1 v_1 \varepsilon^{- 2} L_{\varepsilon}^{- 1} (P_0+P_2) v_1^2 \\
+\nu_3 P_1 (v_1)^3 - P_1v_1^5
+ \mathcal{O}(\varepsilon)  
+ \partial_T W_1\;.
\end{multline}
\subsection*{Identifying the limit:} 
Using the ansatz
\[
v_1(T,X) = A(T,X)e^{iX/\varepsilon} + \text{c.c.} 
\]
we see that we can treat almost all terms in \eqref{e:final5} 
in exactly the same way as in~\eqref{sec:limiting}.
Only the term $P_1v_1^5$ was not present there.
Here we obtain similar to the cubic case
\[
 P_1 (v_1)^5 \approx  10 A|A|^4  e^{iX/\varepsilon} + \text{c.c.} 
\]
and the final result is thus
\[
\partial_T A = 4 \partial_X^2 A + \Big(3 \nu_3 + \frac{38}{9} \nu_2^2\Big) A|A|^2 - 10 A|A|^4 + \eta.
\]

Due to the presence of the quintic nonlinearity,
even with noise, we expect the solutions to be confined with high probability in a neighbourhood of $0$,
but the behaviour is quite different depending on whether the coefficient in front of the cubic is positive or negative.

For a negative coefficient we expect solutions to be on the order of noise-strength close to $0$.
But, if the coefficient is positive, 
then  in the deterministic case there are many stationary solutions $A$, for example the constants defined by  
$10 |A|^2  \approx (3 \nu_3 + \frac{38}{9} \nu_2^2)$,
and in the stochastic case we expect solutions 
to be with high probability of order noise-strength 
centred around these stationary solutions.
 
 %%%%%%%%%%%%%%%%%%%%%%%%%%%%%%%%%%%%%%%%%
\section{Summary and Conclusions}
%\label{sec:summary}

We study as an example the stochastic Swift-Hohenberg equation with additive Gaussian white noise 
in an unbounded domain, where the deterministic equation exhibits a supercritical pitchfork bifurcation.
 
As our main result we studied the reduction of dynamics via 
modulation equation to a stochastic Ginzburg-Landau equation. 
Here we identify the order where small noise first impacts the dynamics close to a change of stability.

Due to the quadratic and cubic nonlinearities in the Swift-Hohenberg equation there are many nonlinear interactions 
of Fourier modes, that lead to a cubic nonlinearity in the limiting Ginzburg-Landau equation.
Surprisingly, the stochastic effects mostly cancel out due to averaging effects and only an additive forcing survives in the limiting equation. 

This result is the first step towards a better 
understanding of the behaviour of the stochastic 
bifurcation for SPDEs with additive translation invariant 
noise on an unbounded domain. Nevertheless, in the setting we are working in here, a full description in terms of random attractors or invariant measures is not yet available, even for the Ginzburg-Landau equation.

\section*{Acknowledgments}  
The authors would like to thank Edgar Knobloch for pointing out this question, as well as the reviewers and the editor for the detailed and helpful feedback on the first version of this paper.
Both authors also acknowledge the support of MOPS at the University of Augsburg.
LAB would like to thank the Hausdorff Institute for Mathematics in Bonn, where part of the research was conducted during the Junior Trimester Program ``Randomness, PDEs and Nonlinear Fluctuations''. 
  
\newcommand{\noopsort}[1]{}

 \end{document}